\documentclass{amsart}
\usepackage{enumerate,amssymb,amsmath,graphicx,ucs, tikz, mathtools}
\usepackage{wrapfig}
\usepackage{hyperref}
\usepackage{todonotes}
\usepackage[cp1251]{inputenc}
\usepackage{times,mathabx,amsfonts,amssymb,amsrefs,mathrsfs,tikz, xcolor}
\usepackage{todonotes}
\usepackage{mathabx}

\usetikzlibrary{decorations,decorations.pathmorphing}

\usetikzlibrary{matrix}
\usetikzlibrary{math, patterns}

\numberwithin{figure}{section}

\theoremstyle{plain}

\newcommand{\thistheoremname}{}

\newtheorem*{genericthm*}{\thistheoremname}
\newenvironment{namedthm*}[1]
{\renewcommand{\thistheoremname}{#1}%
	\begin{genericthm*}}
	{\end{genericthm*}}

\newtheorem{thm}{Theorem}[section]

\newtheorem{lemma}[thm]{Lemma}

\newtheorem{claim}[thm]{Claim}

\newtheorem{corol}[thm]{Corollary}

\theoremstyle{definition}
\newtheorem{definition}[thm]{Definition}
\newtheorem{rem}[thm]{Remark}
\newtheorem*{thm*}{Theorem}
\newtheorem*{thm2*}{Informal Theorem}
\newtheorem*{definition*}{Definition}
\newtheorem*{lemma*}{Lemma}
\newtheorem*{statement*}{Claim}

\newtheorem*{corol*}{Corollary}
\newtheorem*{question*}{Question A}
\newtheorem{question}{Question}

\def\q#1.{{\bf #1.}}

\newcommand{\diam}{\operatorname{diam}}

\long\def\comment#1{}

\author{Anna Erschler}
\address{A.E.:  
C.N.R.S., LPSM, Sorbonne University, CNRS, Paris, France}
\email{anna.erschler@sorbonne-universite.fr }

\author{Ivan Mitrofanov}
\address{I.M.: Geneva University}
\email{phortim@yandex.ru }

\title{Finite dimensional amenable groups}

\date{21.08.2025}



\begin{document}
\begin{abstract}
We show that an amenable group of finite Assouad-Nagata dimension satisfies
the property $H_{FD}$ of Shalom. Such infinite groups are known to admit a virtual homomorphism onto $\mathbb{Z}$, and thus our result implies that an amenable group of finite $AN$-dimension cannot be a simple group. We can also conclude that an amenable group of finite $AN$-dimension cannot be a torsion group.
Our proof is based on new estimates of diameters of F{\o}lner couples.
 We  prove that any amenable group of finite $AN$-dimension  admits F{\o}lner couples inside balls of linear diameter and more generally estimate the radius of the balls containing F{\o}lner couples in groups of finite asymptotic dimension.
 This result strengthens the result of Nowak about diameters of F{\o}lner sets.
\end{abstract}
\maketitle

\section{Introduction and statement of the main result}

This paper is intended to be the first in a series of papers that study the diameters  of F{\o}lner sets and their relation to other invariants of groups, including those related to notions of dimension.

Given a number $r>0$, a collection of subsets of a metric space is said to be {\it $r$-disjoint} if the distance between any two distinct subsets in the collection is greater than $r$. 

\begin{definition} \label{def:nagata} {\it Assouad-Nagata dimension} of a metric space is defined as follows. 
Let $M$ be a metric space, let  $d$  be a non-negative integer and let $K\geqslant 1$.
Suppose that for any $r>1$ there exists a partition of $M$ into a collection of subsets colored with $d+1$ colors such that diameters of sets are at most  $K r$ and such that for each color the collection of sets of this color is $r$-disjoint.
In this case, 
we say that the Assouad-Nagata dimension
of $M$ is at most $d$. 
	\end{definition}
We will also refer to  Assouad-Nagata dimension as $AN$-dimension.  Another terminology for spaces of finite $AN$-dimension is to say that the spaces have the Higson
property, see e.g \cite{nowak}.
Given a metric space $M$, we take the minimum of $d$ such that the $AN$-dimension of $M$ is at most $d$, and we say that the $AN$-dimension of $M$ is, by definition, equal to this minimum (a nonnegative integer or infinity).

There are several equivalent definitions of $AN$-dimension, see \cite{BellDranishnikovBedlewo}.
This notion is essentially due to Nagata in his works in the 50s and 
	goes back to his paper \cite{Nagata58}. Assouad has studied this notion in the early eighties 
	and the term {\it Nagata dimension} is used in  his paper \cite{Assouad82}. 
	If a space is uniformly discrete, as it happens for finitely generated groups endowed with word metric,  only large scales
	and global properties matter; the notion coincides in this case with {\it linearly  controlled metric dimension}, and is also called
	{\it linearly controlled asymptotic dimension}. For uniformly discrete
	spaces $AN$-dimension is a quasi-isometric invariant.

We recall in the next section various	 
 spaces and groups which are known to have finite $AN$-dimension. 
 There are many such groups both among amenable ones, as well as among non-amenable ones.
 Virtually nilpotent groups have finite $AN$-dimension, and there are many solvable groups of exponential growth with finite $AN$-dimension.

A group has  the property $H_{\rm FD}$ of Shalom
if every unitary $G$-representation $\pi$ with non-zero reduced cohomology
$\bar{H}^1(G,\pi)$ admits a finite-dimensional sub-representation.
See Subsection \ref{subsection:Shalom} for more details on this definition.
It is known that an infinite finitely generated amenable  group with the property $H_{\rm FD}$ admits a  finite index subgroup with an infinite abelianization; and hence such subgroup admits a surjective homomorphism to $\mathbb{Z}$, this is proven by Shalom in   Theorem 4.3.1 in \cite{shalom}.

The main result of this paper is
\begin{namedthm*}{Theorem A}
Each finitely generated amenable group of finite $AN$-dimension satisfies property $H_{\rm FD}$ of Shalom. In particular, each such group admits a virtual surjective homomorphism to $\mathbb{Z}$.
\end{namedthm*}

We mention here that even for very basic constructions it might be difficult to check whether the group has property $H_{\rm FD}$; for more on this, we refer again to Subsection \ref{subsection:Shalom}.

Theorem A implies the following.
\begin{corol}\label{corol:simplegroups}
There are no simple groups among finitely generated infinite amenable groups of finite $AN$-dimension.
\end{corol}

Note that the assumption that the group is amenable in Corollary \ref{corol:simplegroups}
above is essential. Indeed, observe that  simple groups of Burger and Mozes 
\cite{burgermozes} are quasi-isometric to a product of two trees and thus
have finite $AN$-dimension (equal to two).

Theorem A also implies
\begin{corol}\label{corol:torsion}
There are no torsion groups among finitely generated infinite amenable groups of finite $AN$-dimension.
\end{corol}
To our knowledge, the situation in the non-amenable case is not known: see Question \ref{question:burnside} in Section \ref{sec:questions}.

Now we comment on possible developments of Theorem A. 
The main step in the proof of the polynomial growth theorem is to prove that a group of polynomial growth admits a finite index normal subgroup which maps onto $\mathbb{Z}$, see more on this in Subsection \ref{subsection:growth}.
If we have a group of  superpolynomial growth of finite $AN$-dimension, then 
conjecturally (see Claim \ref{claim:fromGrigorchuk}) such group has exponential growth. In contrast with the polynomial growth case, even if we know that the group admits a surjective homomorphism to $\mathbb{Z}$, it seems to be a difficult task to control the kernel of such a homomorphism and to get information about the algebraic structure of our group. Still, in view of the claim of our Theorem, we are inclined to ask
\begin{question*} 
Is every finitely generated amenable group of finite $AN$-dimension elementary amenable?
\end{question*}

A positive answer to this question would be a wide extension of the polynomial growth theorem of \cite{gromovpolynomial}.
Indeed, if one knows that the group is of polynomial growth and is elementary amenable, an elementary argument proves that if the group is not virtually nilpotent,
then its growth would be exponential, see \cite{chou}, Theorem 3.2.
That is the theorem that proves that the growth of an elementary amenable group is either polynomial or exponential.

\bigskip

The paper has the following structure.
In the next section, we discuss the definition of Shalom's property, of invariants of amenable groups including sizes of  F{\o}lner sets and F{\o}lner couples, and some background on asymptotic dimension.

In Section 3 we formulate and prove Theorem \ref{pr:dimtoisodiam},
which is more general than the first claim of Theorem A.
Theorem \ref{pr:dimtoisodiam} bounds the diameters of F{\o}lner couples for groups of finite asymptotic dimension from above. 
This extends a result of Nowak, who proved similar upper bounds for diameters of F{\o}lner sets. At the end of this section, we also discuss various smallness conditions of amenable groups, and explain that finiteness of $AN$-dimension, which can be also considered as such a smallness condition, implies other known smallness conditions. 
A Lemma 4.5 in
\cite{erschlerzheng} 
together with a well-known connection between $l_2$-profile and F{\o}lner couples
allows us to deduce the cautiousness property of random walks under the assumption of the existence of controlled F{\o}lner couples. 
We explain the cautiousness condition in Section 3. This cautiousness condition is a sufficient condition for property $H_{\rm FD}$, see \cite{erschlerozawa}.
Theorem A follows therefore from
Theorem \ref{pr:dimtoisodiam}.

In the last section, we discuss several open questions about groups of finite dimension. We discuss current knowledge and relevant open problems related to Question A of the introduction.

\vskip12pt
 {\bf Acknowledgements.} 
This work of the first-named author has received funding from  
ANR PLAGE, 
the Marvin V. and Beverly J. Mielke Fund and EPSRC grant no EP/R014604/1.34;
she thanks the Isaac Newton Institute for Mathematical Sciences, Cambridge, for support and hospitality during the programme "Operators, Graphs, Groups". 
Second-named author 
 is supported by SNSF grant TMAG-2\_216487\slash1.

\section{Definitions and background}
\subsection{$AN$-dimension and asymptotic dimension of metric spaces and groups }

Given a finitely generated group $G$ and its finite symmetric generating set $S$, we consider the word metric $d_S$, 
 which is invariant under multiplication from the left.
 We recall that  
 the word length $l_S$ is defined as  the minimum of $l$, such that
 $g=s_1 s_2 \dots s_l$ for some $s_1, s_2, \dots s_l\in S$ and  $d_S$ is defined as 
 $d_S(h, g)=l_S(h^{-1}g )$.

Different sets of generators give rise to bi-Lipschitz equivalent (and in particular quasi-isometric) word metrics. 
When the choice of the generating set $S$ is not important, we omit $S$ in the notation for
$l_S$ and $d_S$. When we speak about $AN$-dimension of a group, we consider the group as a metric space endowed with word metric. 

Spaces with doubling property \cite{LangSchlichenmaier} are known to have finite $AN$-dimension. Thus, all virtually nilpotent groups have this property. 
A more general class of groups of finite $AN$-dimension is that of polycyclic groups \cite{HigesPeng}, Corollary 5.6.
Another class
of amenable groups 
of finite $AN$-dimension is that of 
wreath products of groups of linear growth with finite groups \cite{BrodskiyDydakLang}.

Graphs and groups that admit quasi-isometric embedding into finite products of trees 
are known to have finite $AN$-dimension, see \cite{BuyaloDranishnikovSchroeder}. This class contains many non-amenable graphs, for example all hyperbolic groups and $\delta$-hyperbolic spaces that are "doubling in the small" \cite{LangSchlichenmaier},
but also various examples of amenable groups, including solvable Baumslag-Solitar groups, see e.g.
example 7.3 in \cite{nowak}.

Now recall a more general definition of {\it asymptotic dimension}. 
This definition is given by Gromov in \cite{gromovasymptotic}, and the theory was then developed in the works of Dranishikov and other researchers (see e.g.
\cite{BellDranishnikovBedlewo} and references therein). As we have mentioned speaking about the definitions of $AN$-dimension, there are also several equivalent ways to define this notion, and we refer again to \cite{BellDranishnikovBedlewo}.

The definition below is formulated analogously to that of $AN$-dimension that we have discussed in the introduction, but the assumption that the sets have linear diameters  in $r$ is replaced by a weaker condition about a uniform upper bound for these diameters:
\begin{definition}
Let $M$ be a metric space and $d $ be a non-negative integer.
Suppose that there exists $f_M:\mathbb{R}_+\to\mathbb{R}_+$ such that for any $r>1$
there exists a partition of $M$ into sets colored into $d+1$ colors such that diameters of sets are $\leqslant f_M(r)$, and for any color the collection of sets of this color is $r$-disjoint. 
Then we say that $M$ has {\it asymptotic dimension} at most $d$. 
\end{definition}

\begin{claim}\label{rem:Hd} Let $G$ be a finitely generated group. 
	If for each  $d$ there exists   an infinite finitely generated group $H$, such that $H^d$ is a subgroup of $G$, then the 
	asymptotic dimension of $G$ 
	is infinite. 
\end{claim}
For example, wreath product $\mathbb{Z} \wr \mathbb{Z}$ and many other solvable groups have infinite asymptotic dimension. 
Many solvable groups, such as, for example, (locally finite)-by-Abelian ones, have finite asymptotic dimension. 
For some groups, the definition of $AN$-dimension and asymptotic dimension differ. 
Inside amenable groups, one can witness this phenomenon, for example, for wreath products $\mathbb{Z}^d \wr B$, $d\ge 2$.  The asymptotic dimension of these groups is finite when $B$ is  finite, but $AN$-dimension is infinite \cite{BrodskiyDydakLang}.

\subsection{Shalom's property} \label{subsection:Shalom}
Shalom defines his property for discrete groups, and many of his results are under a stronger assumption of finite generation of groups. This will be the context of our paper.
Let $G$ be a finitely generated group,
and $\pi$ be a unitary representation
of $G$ on a Hilbert space $\mathcal{H}$.
A function $b : G \to \mathcal{H}$ is a {\it 1-cocycle} if
$b(gh) = b(g) + \pi_g b(h)$ for all $g, h$ in $G$.
The vector space of
1-cocycles is denoted as $Z^1(G, \pi)$.
We also recall that 
$b : G \to \mathcal{H}$
is a  {\it 1-coboundary} if there exists $v$ in
$\mathcal{H}$ such that 
$b(g) = v-\pi_g v$ 
for all $g \in G$.
We denote by $B^1(G, \pi)$ the subspace of 1-coboundaries.
We consider the closure of  $B^1(G, \pi)$ in the topology corresponding to 
 pointwise convergence for coboundaries, and denote it  by $\overline{B^1}(G, \pi)$.

The {\it first reduced cohomology group} of a representation $\pi$ is by definition the quotient space
$$
\overline{H^1}(G, \pi) = Z^1(G, \pi)/\overline{B^1}(G, \pi).
$$
See Chapter 3 of \cite{bekkadelaharpevalette} for an introduction to first reduced cohomology groups.

As we have mentioned in the introduction, a finitely generated group $G$ has the property $H_{\rm FD}$ if  any unitary representation $\pi$ with non-trivial
$\overline{H^1}(G, \pi)$ admits a finite-dimensional subrepresentation (see \cite{shalom}, definition on page 125).

This property  is stable under taking direct products, central extensions \cite{shalom} and moreover under taking FC-central extensions. Nilpotent groups have a stronger property $H_T$ which says that any unitary representation with non-trivial first reduced cohomology has a trivial subrepresentation
(see Thm 4.1.3 in \cite{shalom}). 
The property $H_{\rm FD}$ is stable under quasi-isometries inside the class of amenable groups. It is an open question whether $H_{FD}$ is also preserved by quasi-isometric embedding of
non-amenable groups, and even the question of Shalom whether a group which is 
quasi-isometric to a group with property T of Kazhdan satisfies $H_{FD}$ is open (see
\cite{shalom}, page 180).

All polycyclic groups have this property (Theorem 1.13 in \cite{shalom}), some without satisfying $H_T$.  
A one-dimensional lamplighter group $\mathbb{Z}\wr (\mathbb{Z}/p\mathbb{Z})$ has $H_{\rm FD}$ \cite{shalom}, $d$-dimensional lamplighter groups $\mathbb{Z}^d \wr \mathbb{Z} /p \mathbb{Z} =    \left( \mathbb{Z} /p \mathbb{Z} \right)^{\mathbb{Z}^d} \rtimes \mathbb{Z}^d$
, $d\ge 3$  violate $H_{\rm FD}$ \cite{brieusselzhengHFD}, and for $d=2$ it seems unknown whether lamplighter groups satisfy $H_{FD}$.
So even for very basic examples of solvable groups it might be difficult to check whether the group has $H_{FD}$. Shalom suggested that a finiteness of Hirsch number can characterize solvable groups with $H_{FD}$, but counterexamples in both directions are constructed in \cite{brieusselzhengHFD}.

If the group has $H_{\rm FD}$ property of 
Shalom, there are clearly three  possibilities. 
\begin{enumerate}
\item $G$ is an amenable group with $H_{\rm FD}$.
\item $G$ is a (non-amenable) group without any 
unitary represenations with non-trivial $\overline{H^1}$.
\item $G$ is non-amenable group, that admits some unitary reperesentations with non-trivial $\overline{H^1}$ (which admit finitely dimensional subrepresentation)
\end{enumerate}
The existence of a representation with non-trivial 
$\overline{H^1}(G, \pi)$ is well-known to be equivalent to the violation of property $T$ of Kazhdan (see e.g.
Section 3 of \cite{bekkadelaharpevalette}).
In particular, groups mentioned in (2) are exactly the class of groups with property $T$.
We mention that some of these groups, for example
hyperbolic groups with property $T$, have finite $AN$-dimension. We do not 
recall the notion of Kazhdan property T or hyperbolic groups.
Our paper is about amenable groups. We mention the information above
to give some perspective for questions about non-amenable groups with property $H_{\rm FD}$, mentioned in the introduction and in the last section of this paper.

\subsection{Random walks on groups and functions defined by such random walks}

Now we recall several functions related to random walks on groups.
Given a probability  measure $\mu$ on a group, the {\it random walk} $(G,\mu)$  is the Markov chain on $G$ with transition probabilities $P(x, xh)= \mu(h)$. 
We will assume that the random walks we mentioned are defined by measures $\mu$, which are symmetric ($\mu(g) = \mu(g^{-1})$), that the support  of $\mu$ is finite and generates $G$. The random walks defined by such measures are called {\it simple} random walks.
We do not necessarily assume that the support of the $\mu$ is the generating set $S$ that we use for the definition of the word metric.
We denote by $W_n$ the position  of the random 
walk at time $n$. We assume that the random walk 
starts at the identity, that is $W_0=e$.

\begin{definition} {\it Drift function} of a random walk $(G,\mu)$  measures the average distance to the origin after $n$ steps of a random walk:
$$
L(n) = E[l(W_n)]. 
$$

\end{definition}

\begin{definition}
A random walk is said to be {\it diffusive} if $L(n) \le C \sqrt{n}$ for some $C>0$ and all $n$.

\end{definition}

A similar notion is  cautiousness. In that definition, we do not search for an upper bound for the expected length $l(W_n)$, but a bound from below of the probability for $W_n$ to stay inside  a ball of small radius, that is, of radius $ \varepsilon (\sqrt{n})$:

\begin{definition}
A random walk $G(\mu)$ is said to be {\it cautious} if for all $\varepsilon>0$ there exists $p_{\varepsilon} >0$ and an infinite sequence $n_i$
such that for any $i$ the probability that $\max_{k: 1\le k \le n_i}l(W_{k}) \le \varepsilon \sqrt{n_i} $ is at least $p_{\varepsilon}$.
\end{definition}
This definition is given in \cite{erschlerozawa}, where in Corollary 2.5 it is shown that cautiousness is a sufficient condition for the property $H_{FD}$ of Shalom.

We will also speak about total cautiousness, when this property is required to hold for all $n$, rather than for a subsequence. 

\begin{definition}
We say that a random walk $(G,\mu)$ is {\it totally cautious} if for 
all $\varepsilon>0$ there exists $p_{\varepsilon}>0$ such that for all  $n \ge 1$ the probability that $\max_{1\le k \le n}l(W_{k}) \le \varepsilon \sqrt{n} $ is at least $p_{\varepsilon}$.

\end{definition}

Clearly, total cautiousness of a random walk implies cautiousness.

\subsection{Invariants of finitely generated amenable groups}

We recall the definition of F{\o}lner sets.
\begin{definition}
Let $G$ be a group, and $S$ be a finite generating set of $G$.
A sequence of finite subsets $F_n \subset G$ is said to be a {\it sequence of F{\o}lner sets} if
$$
 \frac{\# \partial_S F_n}{ \# F_n} \leqslant \frac{1}{n}.
$$
\end{definition}

Here $\partial_S F$ denotes the boundary of a finite set $F$, that is, the set of elements $x\in F$ such that $d_S(x, G\setminus F)=1$;
$\#$ denotes the cardinality of a set.

\begin{figure} 
\centering
\includegraphics[width=0.6\linewidth]
{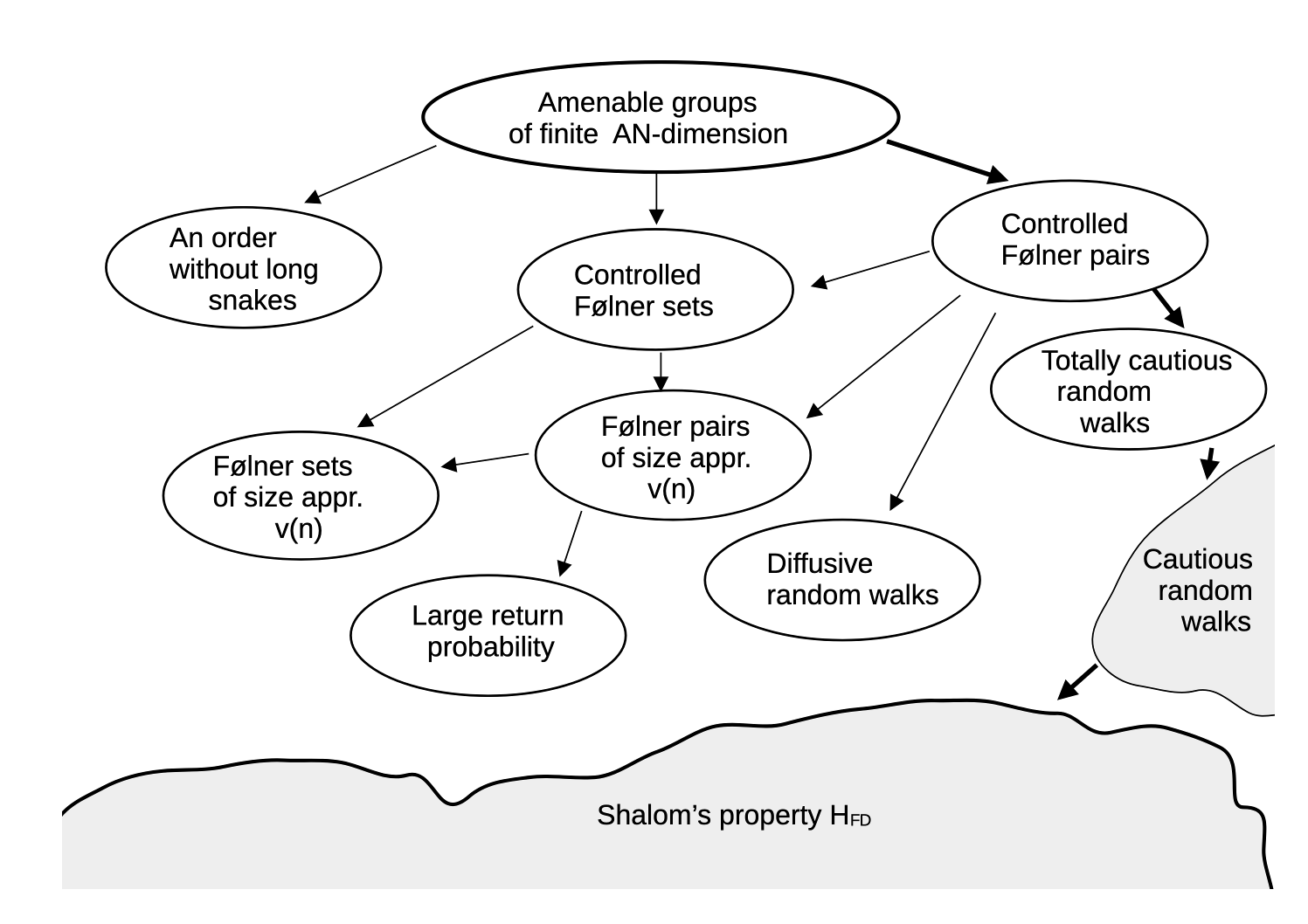}
\caption{\label{pic:smalnessconditions} Smallness conditions for amenable groups. Properties of amenable groups of finite $AN$-dimension.}
\end{figure}

We also recall the definition of F{\o}lner couples.
\begin{definition} \label{defn:couples}
    
A sequence $(F_n',F_n)$ of pairs of finite subsets
of $G$ with $F_n'\subset F_n$ is called {\it a sequence of F{\o}lner couples}
if there is a constant $C < \infty$ such that for all $n$, we have
 \begin{enumerate}
    \item $\#F_n \leqslant C\#F_n'$
    \item $d_S(F_n',G\setminus F_n) \geqslant n$
\end{enumerate}
\end{definition}

F{\o}lner couples are also called F{\o}lner pairs.
This notion was introduced by Coulhon, Grigoryan and Pittet in \cite{coulhongrigoryanpittet}, who observed
that such couples provide lower bounds for the return probability of random walks; see a discussion of anti-Faber-Krahn
inequality in the introduction of \cite{coulhongrigoryanpittet} and Theorem 4.8 in that paper.
In the case when the cardinality of F{\o}lner couples
is asymptotically equivalent to the size of optimal F{\o}lner sets, these bounds are asymptotically optimal,
see Theorem 4.8 and Example 4.2 for particular cases of this phenomenon.
A technical assumption 4.20 in \cite{coulhongrigoryanpittet} is not necessary, we do not give reference for further works  since later in this paper  we will have to consider only the very basic case  of exponential size F{\o}lner couples.
It is easy to see that finitely generated Abelian groups admit F{\o}lner couples, there are groups of exponential growth, such as lamplighters with an Abelian base group, that also have this property.

\begin{definition}
If F{\o}lner sets (or F{\o}lner couples) lie in balls of diameter linear in $n$, they
are called {\it controlled}.
\end{definition}

The notion of controlled F{\o}lner couples
was studied in \cite{tesseracontrolled}, see Definition 4.8
in \cite{tesseracontrolled}, where they are called controlled F{\o}lner pairs.

Now we recall the definition of $l_p$ profile inside the balls.
\begin{definition} Consider $p: 1 \le p \le 2$.
The $l_p$ profile inside the ball of radius $r$ is defined as
\[
\lambda_{S,p}(B(e, r)) = \inf \left\{ \frac{1}{2} \sum_{x, y \in G} \left|f(x) - f(xy)\right|^p \mu(y) : \operatorname{supp} f \subseteq B(e, r) : \|f\|_p = 1 \right\}.
\]
Here the infimum is taken over functions $f$ defined on the ball $B(e, r)$.
\end{definition}

The relation between F{\o}lner couples and $l_p$ profile (not necessarily inside the balls)  is that F{\o}lner couples provide an upper  bound
for this profile.
Indeed, in this case one considers the distance from $F'$ to $G\setminus F$ and this distance function provides this estimate.
Especial interest of this definition when $p=2$
explored in \cite{coulhongrigoryanpittet} and many subsequent papers, is that the existence
of F{\o}lner couples provides the upper bound for the $l_2$ profile, and this in turn, provides the lower estimate for the return probability of this random walk.

\subsection{Growth of groups; comments on the polynomial growth theorem.} \label{subsection:growth}

The growth function of $G$ is by definition the volume of a ball of radius $n$
in the Cayley graph of this group:
$$
v_{G,S}(n)=\# \{ g \in G: l_{S}(n) \le n \}
$$

The growth is clearly at most exponential, so the groups can be of {\it exponential growth}; of {\it polynomial growth}; and otherwise of {\it intermediate growth}. 

Polynomial growth theorem \cite{gromovpolynomial} states that the growth of $G$ is polynomial if and only if the group has a nilpotent finite index subgroup.
As we have mentioned in the introduction, the main step in the proof of this theorem is to prove that a group of polynomial growth admits a virtual homomorphism onto $\mathbb{Z}$: this is the way it was done in the original proof of Gromov \cites{gromovpolynomial}, this is the way it works in Kleiner's proof based on the finite dimension of the space of harmonic functions of linear growth, defined by a finitely supported measure on the group in question  \cite{kleinerpolynomial}, and this is also the argument of Ozawa in \cite{ozawapolynomial}, who realized the strategy suggested earlier by Shalom in proving the Shalom property using the assumption on the polynomial growth.
Indeed, after one knows the existence of a virtual homomorphism to $\mathbb{Z}$, then the proof of the theorem follows by induction on the degree of the polynomial growth: one checks that the kernel of the homomorphism is finitely generated and observes that its growth is at most $n^{d-1}$ if the growth of the initial group is bounded by $n^d$.
If we have a group of  superpolynomial growth of finite $AN$-dimension, then 
conjecturally (see Claim \ref{claim:fromGrigorchuk}) such group has exponential growth.

\section{The diameters of F{\o}lner couples
for groups of finite asymptotic dimensions}

A result of Nowak shows that amenable groups of finite $AN$-dimension admit
controlled F{\o}lner sets; more generally, his result bounds the diameter
of F{\o}lner sets in groups of finite asymptotic dimension.
We claim that moreover the diameters of F{\o}lner couples can also be bounded from above by the control function for the asymptotic dimension, 
that such F{\o}lner pairs are controlled in the case of finite $AN$-dimension. 
This can be used to show that they satisfy several other "smallness" conditions: the return probability of simple
random walks is large, these random walks are diffusive 
and totally cautious. 

\begin{thm}\label{pr:dimtoisodiam}
Assume that a finitely generated amenable group $G$ has finite asymptotic dimension $d$ with the corresponding control function
 $f(n)$.
Then there exists a sequence of F{\o}lner couples $(F'_n,F_n)$ with $C = d+1$ such that
$\diam(F_n) \leqslant f(2n)+2n$.
\end{thm}

\begin{proof}

We recall that by one of the equivalent definitions of amenability, the group $G$ is amenable if it admits a finitely additive, non-negatively valued,  $G$-invariant measure on all subsets, with total measure one.
One can speak about left-invariant means, which are invariant for the multiplication from the left;
about right-invariant means, which are invariant for the multiplication from the right, and about two-sided invariant means. 
Amenability is equivalent to each of these three conditions: existence of left-invariant means, existence of right-invariant means, or existence of two-sided invariant means.

The following elementary lemma might be reminiscent of the mass transportation argument for amenable groups.
\begin{lemma}\label{lem:lemma1}
Let $G$ be a finitely generated amenable group and let $\mu$ be a right-invariant mean on subsets of $G$. 
Consider a family of disjoint subsets $\{B_{\alpha}\}$ of $G$, $\alpha \in X$. Here the set $X$ is  countable or finite.
Assume that  
the diameters of  $B_{\alpha}$, for all $\alpha \in X$, are uniformly bounded.
Consider a family of subsets $A_{\alpha}$, such that $A_{\alpha} \subseteq B_{\alpha}$ for all $\alpha \in X$. Assume that there exists $\lambda>0$ such that for each $\alpha $ it holds $\#B_{\alpha} \geqslant \lambda \# A_{\alpha}$.
Then  the total mean of the union of $A_{\alpha}$ satisfies:
$$
\mu (\bigsqcup A_{\alpha}) \leqslant 1/\lambda.
$$

\end{lemma}

If the set $X$ is finite, the claim of the lemma is straightforward.  We would have
$$
\mu (\bigsqcup A_{\alpha}) = \sum_{\alpha \in X} \mu(A_\alpha) \le
1/\lambda \sum_{\alpha \in X} \mu(B_\alpha) = 1/\lambda 
\mu (\bigsqcup B_{\alpha}) \le  1/\lambda.
$$

We need to prove the lemma in the case when the $X$ is not necessarily finite, and 
$\mu (\bigsqcup A_{\alpha})$ and  $\mu (\bigsqcup B_{\alpha})$ are not necessarily equal to $\sum_{\alpha \in X} \mu(A_\alpha)$ and respectively
to $\sum_{\alpha \in X} \mu(B_\alpha)$.

\begin{proof}
We consider  family of subsets $A_{\alpha} \subset B_{\alpha} \subset G$ as in the formulation of the lemma.
We say  $\alpha \in X$ and $\beta \in X$ are equivalent if there exists
$g\in G$ such that $A_{\alpha}= gA_{\beta}$ and $B_{\alpha}=g B_{\beta}$.
Since the diameters of the sets $B_{\alpha}$ are uniformly bounded, the number of equivalence classes of our equivalence relation on $X$ is finite.
We enumerate  corresponding sets of indexes 
as $C_1,\dots, C_m$. We have $C_i \subset X$ and $\bigsqcup_{k = 1}^{m}C_k = X$.
We consider total means  of the union of $A_{\alpha}$ and $B_{\alpha}$ in a given class $C_i$:
$$
M^A_i = \mu (\bigsqcup_{ \alpha \in C_i} A_{\alpha}), \hspace{0.5cm}
M^B_i = \mu (\bigsqcup_{\alpha \in C_i} B_{\alpha}).
$$
 
Observe that the cardinality of the sets $A_{\alpha}$ is the same for all
 pairs $(A_{\alpha}, B_{\alpha})$ corresponding to the same equivalence class. We 
denote 
by $a_i$ the cardinality of $A_\alpha$ for $(A_{\alpha}, B_{\alpha})$ corresponding to the class $C_i$. Analogously, the cardinality of the sets
$B_{\alpha}$ for $(A_\alpha, B_{\alpha})$ for the same equivalence class are the same.
We denote  by
$b_i$ the cardinality of $B_{\alpha}$ for $(A_\alpha, B_{\alpha})$
corresponding to the class $C_i$.

 By the assumption of the lemma, we know that $b_i \geqslant \lambda a_i$.
 Fix some $\alpha$ in this class $C_i$.  For each $\beta \in C_i$ fix $g_\beta$ such that $A_{\beta} = g_\beta A_\alpha$ and we assume that $g_\alpha=e$.
For each point $h \in A_\alpha$, we can consider its images by translations
 $g_{\beta} h$, for all $\beta \in C_i$.
We denote this set of points $G_i^h$.
Fix some $h \in A_\alpha$, put $G_i= G_i^h$ and consider $M_i^G = \mu(G_i)$.
Observe that for each $h' \in A_{\alpha}$, we have 
$G_i^{h'} = 
\{g_{\beta} h h^{-1}h'\} = G_i h^{-1}h'$.
Since the right-invariant mean $\mu$ is finitely additive and since $A_\alpha$ is finite, we can 
therefore conclude that  $M_i^A = a_i M_i^G$ and $M_i^B = b_i M_i^G$.
 Hence we know that 
$$
1 \geqslant \mu (\bigsqcup_{\alpha \in X}B_{\alpha}) =
\sum_{i = 1}^{m} M^B_i = \sum_{i = 1}^{m} b_i M^G_i \geqslant
$$
 
$$
\geqslant \lambda \sum_{i = 1}^{m} a_i M^G_i = \lambda \sum_{i = 1}^{m} M^A_i = \lambda \cdot \mu (\bigsqcup_{i = 1}^{\infty} A_i),
$$
 and this completes the proof of the lemma.
\end{proof}

Now we  prove  Theorem \ref{pr:dimtoisodiam}.

Fix $n\ge 1$.
We use the definition of the control function  for the asymptotic dimension for $r=2n$.
We know that  we can color the elements of 
 $G$ in  $d+1$ colors and partition the elements of each color into parts such that the following holds. The distance between distinct parts of the same color is greater than $2n$,
 and each part has diameter at most $f(2n)$.

By the assumption of the theorem,  $G$ is amenable. We fix a right-invariant mean on the subsets of $G$.
Observe that there exists at least one color with the mean of the union of parts
of this color
 $\geqslant \frac{1}{d+1}$. Choose one  such color, and 
  denote the parts of this color by $A_{\alpha}$, $\alpha \in X$, here $X$ is either a countable or finite set.

By our assumption on their diameter, we know that the 
   $n$-neighborhoods  $B(A_{\alpha}, n)$ are pairwise non-intersecting.

We apply Lemma \ref{lem:lemma1} to the parts  $A_{\alpha}$ of the fixed color, defined
above,
and to their neighborhoods  $B_{\alpha}=B(A_{\alpha}, n)$, $n \in \mathbb{N}$.
We conclude that for some $\alpha$, it holds $\#B(A_{\alpha},n) \leqslant (d+1)\#A_\alpha$.

We  put $F'_n = A_{\alpha}$ and $F_n = B(A_\alpha, n)$.
We know that for all $n$
$$
\frac{\# F_n}{\# F'_{n}} \le d+1,
$$
and thus $(F'_n, F_n)$ satisfies the first property in the definition of F{\o}lner couples (see Definition \ref{defn:couples}).

Since, by construction, $F_n$ is a $n$-neighborhood of $F'_n$, it is clear that the distance 
$$
 d (F_n', G \setminus F_n) \ge n,
$$
and thus $(F'_n, F_n)$ also satisfies the second property of Definition \ref{defn:couples}. Therefore  $(F'_n, F_n)$ forms indeed a sequence  of F{\o}lner couples.

Finally, observe that the diameter of $F_n = B(A_\alpha, n)$ is at most $2n$ plus
the diameter of $A_{\alpha}$, and hence the diameter of $F_n$ is at most
$f(2n)+2n$. This completes the proof of the Theorem.
 
\end{proof}

\section{Proof of Theorem A. Properties of groups of finite Assouad-Nagata dimension}

In a particular case when  not only asymptotic dimension but also $AN$-dimension of $G$ is finite, Theorem \ref{pr:dimtoisodiam} tells us that
the group admits controlled F{\o}lner sets. Indeed, if the control function is at most $Kn$, then the group $G$ admits F{\o}lner couples of diameter at most $(K+2)n$. And controlled F{\o}lner pairs in a group imply property $H_{\rm FD}$ of Shalom, as it was shown
in \cite{erschlerzheng} using the previous criterion of \cite{erschlerozawa}.
We remind below this conclusion and also explain  that Theorem \ref{pr:dimtoisodiam} and previously known implications between smallness conditions,  we can conclude that amenable groups of finite $AN$-dimension
also satisfy the following "smallness" conditions.

Such known implications between smallness conditions are shown in Figure \ref{pic:smalnessconditions}. Thick arrows at the right side of the picture correspond to our proof of Theorem A. The first arrow in this path of rays is a particular case,
under an extra assumption that $AN$-dimension is finite, of the claim of our Theorem  3.1. 
Other thick arrows indicate already known results needed for the conclusion of Theorem A.

\begin{corol}
Let $G$ be a finitely generated amenable group of finite $AN$-dimension.
Then
\begin{enumerate}

\item $l_p$-profile of $(G,S)$ inside the balls is small, that is,  for all $r$ we have
$$
\lambda_{S,p}(B(e, r)) \le Cr^{-p}
$$
\item Simple random walks on $G$ are totally cautious.

\item The return probability of simple random walks on $G$ is large. That is, for all $n$
it holds
$p_n(e,e) \ge \exp(- {\rm Const}n^{1/3})$.

\item Simple random walks on $G$ are diffusive.
\end{enumerate}

\end{corol}

A comment about return probability: in the proof below we will use the fact that the growth function is at most exponential. If the growth is subexponential,  the lower bound can be improved in terms of growth.

\begin{proof} 

1)  From Theorem 
\ref{pr:dimtoisodiam} we see that in case of a linear control function, we have F{\o}lner couples inside balls of linear diameter. In other words, we have controlled F{\o}lner couples. And
controlled  F{\o}lner couples imply the upper bound of $l_2$-profile inside the balls (of linear radius),
see \cite{tesseracontrolled}, Proposition 4.9.

2) We need to explain the result of \cite{erschlerzheng}
that estimates of 1) for  the $l_2$-profile imply the total cautiousness estimates.
The statement of  Lemma 4.5 in \cite{erschlerzheng} affirms cautiousness.  But in fact, the proof of that lemma, in the case when we have estimates for $l_2$-profile inside the balls of radius $r$, for all $r$, shows total cautiousness.

 Indeed, in the last line of the proof
of Lemma 4.5 in \cite{erschlerzheng} it was proved the following.
Assume that $(r_i)$ is a sequence of numbers such that for all $i\in \mathbb{N}$ the $l_2$-profile inside the balls  satisfies 
$$
\lambda_{S,2}(B(e, r_i)) \le Cr_i^{-2}.
$$
Then for any $c>0$
and $n_i =  (r_i/c)^2$
the following holds: the probability that the random walk at positions $W_k$, $k: 1 \le k \le n_i$,  stays in the ball of radius $2c\sqrt{n_i}$
satisfies
\[
\mathbb{P} \left( \max_{k \leq n_i} l_S (W_k) \leq 2c\sqrt{n_i} \right) 
\geq \delta(c, C) > 0.
\]
 
 Putting $r_i=i$ we can conclude the total cautiousness of our random walk.

3) Estimates of return probability in the presence of F{\o}lner couples is a result
of \cite{coulhongrigoryanpittet}, Theorem 4.8.
As we have mentioned, this was the motivation for the notion 
of F{\o}lner pairs in that paper.

4) We show that 1) implies 4). We use 1) for $p=2$, so that we know the estimates of $l_2$-profile.
Then the diffusive upper bounds for the drift function $L(n)$ in a particular case of the claim of Theorem 1.6 \cite{pereszheng}, for
$\alpha=1$ and $\theta=2$ in the formulation of that theorem.

\end{proof}

\begin{question*}
Are  "smallness conditions" in the ovals of Figure \ref{pic:smalnessconditions} equivalent? 
\end{question*}
To our knowledge, it is known for none of the pairs of distinct  properties mentioned in the question. As to  grey non-oval domains in the picture, these classes of groups are essentially larger than those mentioned in the ovals. And the class of groups with  Shalom's property
(depicted in the large grey domain at the bottom of the picture) is also essentially larger than groups with cautious simple random walks
(a smaller grey domain).

See also Question 6.1 \cite{erschlerzhengLLN} for a discussion of the equivalence of some of the  conditions on our picture, and some conditions not included in our picture. As we have mentioned, arrows between ovals signify that one property implies the other one. We also mention a recent result of Hutchcroft that shows that under some version of cautiousness (referred to by the author as a version of diffusiveness), large return probability is equivalent to the small growth of the F{\o}lner function. 
While many properties in ovals describe smallness of well-known invariants of groups, one of the properties  speaks
about "snakes", a sequence of points in  an ordered metric space,
which alternates between $r$-neighborhoods of two points.
A recent invariant introduced in \cite{ErschlerMitrofanov2}
speaks about a possibility to order the space so that it does not admit a sequence of snakes with the ratio of diameter to $r$ tending to infinity.
This property is shown in that \cite{ErschlerMitrofanov2} to be a corollary of finite $AN$-dimension.

As to Shalom's property $H_{FD}$, depicted as a large area in the bottom of the picture, we mention that the class of amenable groups with $H_{FD}$
is much larger than these classes of "small groups". Using two different constructions, it is shown in \cite{brieusselzhengHFD} and \cite{erschlerzheng}
that groups with Shalom's property can have arbitrarily large F{\o}lner sets, and thus arbitrarily quick subexponential probability of return to the origin; and an  arbitrarily large sublinear drift. It is also shown in \cite{brieusselzhengHFD} that groups with this property can contain as a subgroup a sum of infinitely many copies of $\mathbb{Z}$ (see Proposition 3.6 of \cite{brieusselzhengHFD} and its proof), and thus infinite $AN$-dimension, and even infinite asymptotic dimension.

\section{Remarks and open questions} \label{sec:questions}

To have a  positive answer to  Question A, one  needs to exclude in particular a possibility to have groups of intermediate growth among groups of finite $AN$-dimension. 
Even in this formulation, it seems to be a very difficult challenge.
In this section, we discuss  known results  and open questions about the dimension of groups on intermediate growth and the dimension of non-elementary amenable groups.

Grigorchuk's Gap conjecture for isoperimetry states that	F{\o}lner function of any finitely generated group is either polynomial or at least exponential, see \cite{grigorchuksurvey}, Conjecture 12.3 (ii). This conjecture was also formulated in an earlier unpublished manuscript (2000) of Grigorchuk and Pansu. 

\begin{claim}   \label{claim:fromGrigorchuk} 
If Grigorchuk-Pansu's Gap conjecture for isoperimetry holds true, then  each group $G$ of intermediate growth has infinite $AN$-dimension.
\end{claim}
\begin{proof} 
If this conjecture holds true, then for any group  of intermediate growth, the F{\o}lner function is asymptotically strictly larger than the growth function. Thus under the assumption of the claim we know that $G$ does not admit controlled F{\o}lner sets. 
On the other hand, if the $AN$-dimension  of $G$ is finite,
then Thm 5.3 in \cite{nowak} implies that $G$ admits F{\o}lner couples. We conclude therefore that the $AN$-dimension of $G$ is infinite.
\end{proof}

All known examples of groups of intermediate growth have infinite asymptotic dimension:
\begin{rem}\label{rem:detailedintermediate}
All known groups  $G$ of intermediate growth  satisfy the following.
For each $d\ge 1$ there exists an infinite finitely generated group $H$ such that  $G$ contains
$H^d$ as a subgroup.
To our knowledge, it is also the case for all known amenable, non-elementary amenable groups.
\end{rem}

For some groups, the group $H$ in the Remark above can be chosen to satisfy the claim for all $d$, as for example for the first Grigorchuk group. This group has an even stronger property of being self-similar, that is, $G$ is commensurable with $G^2$, and this self-similarity is used in \cite{smith} to claim infinite asymptotic dimension.
For some other $G$, the choice of group $H$ can depend on $d$; this is the case
for Grigorchuk groups $G_w$ \cite{grigorchukdegrees}, defined by one-sided infinite sequences of the symbols $0$, $1$, $2$, which 
have intermediate growth unless two symbols appear only a finite number of times. If $w$ is not periodic, the dependence of $H$ on $d$ is essential.
We also mention many interesting examples of self-similar groups, with rich families of groups of exponential growth, and also with families of groups of intermediate growth, can be found among IMG (iterated monodromy groups). Some of these groups have intermediate growth, and some are not elementarily amenable groups of exponential growth.
An intriguing class of groups of intermediate growth is constructed several years ago by Nekrashevych  \cite{nekrashevych} (such groups provide
his examples of simple groups of intermediate growth).
In contrast with Grigorchuk or IMG groups, where self-similarity or weak self-similarity comes from the action of
groups on a rooted tree, some groups of Nekrashevych are
simple, and in particular cannot act on a rooted tree.
But all such groups satisfy the assumption of the Claim \ref{rem:Hd},  and hence their  asymptotic dimension is infinite.

\begin{question} Does each group of intermediate growth have infinite asymptotic dimension?
\end{question}

Another question in this context is whether converse statements to Nowak's result
about groups of finite $AN$-dimension hold true:
\begin{question}
	Is it true that if the F{\o}lner function of a group of exponential growth  is exponential, then $AN$-dimension is finite?
    Is it true that if F{\o}lner function is equivalent to the growth function, then $AN$-dimension is finite?
\end{question}

If the answer to the above-mentioned Question 1 about groups of intermediate growth  is positive, then the latter can happen only if the group is of polynomial or exponential growth. We mention in this context that it seems interesting to know whether 
the  F{\o}lner function of Grigorchuk groups is strictly larger than exponential. So far it is only known that this function is  at least exponential \cite{erschlerzheng}. We also mention that a well-known gap problem of Grigorchuk asks whether the F{\o}lner function of any not virtually nilpotent group is at least exponential.

One can moreover ask: is it true that the F{\o}lner function of any group of 
intermediate growth is asymptotically strictly larger than exponential?

\begin{question}
	If the F{\o}lner function is at most exponential, is it true that $AN$-dimension coincides with the asymptotic dimension?
\end{question}

We know from Corollary \ref{corol:torsion} that it makes sense to ask the following 
only for non-amenable groups:
\begin{question}\label{question:burnside}
Does there exist an infinite
torsion group of finite $AN$-dimension?
\end{question}
To our knowledge, it is not known whether the asymptotic dimension or  $AN$-dimension of free Burnside groups is finite or infinite.

\end{document}